\documentclass[a4paper,10pt]{article}
\usepackage[utf8x]{inputenc}
\usepackage{amssymb}
\usepackage{multicol}
\usepackage{enumitem}
\usepackage{amsmath}
\usepackage{amsthm}
\usepackage{mathrsfs}
\usepackage{a4wide}
\usepackage{cite}
\usepackage[english]{babel}
\usepackage{nicefrac}

\title{Quantitative results for Halpern iterations of nonexpansive mappings}
\author{ Daniel K\"ornlein\thanks{The 
author has been 
supported by the German Science Foundation (DFG 
Project KO 1737/5-2).}\\[0.2cm]
Department of Mathematics \\  Technische Universit\"at Darmstadt\\ 
Schlossgartenstra\ss{}e 7, 64289 Darmstadt, Germany}

\newtheorem{theorem}{Theorem}[section]
\newtheorem{proposition}[theorem]{Proposition}
\newtheorem{corollary}[theorem]{Corollary}

\theoremstyle{definition}

\newtheorem{lemma}[theorem]{Lemma}
\newtheorem{definition}[theorem]{Definition}

\theoremstyle{remark}

\newtheorem{remark}[theorem]{Remark}

\newcommand{\NN}{\mathbb{N}}
\newcommand{\RR}{\mathbb{R}}

\begin{document}

\maketitle

\begin{abstract}
We give a rate of metastability for Halpern's iteration relative to a rate of metastability for the resolvent for nonexpansive mappings in uniformly smooth Banach spaces, extracted from a proof due to Xu. In Hilbert space, the latter is known, so we get an explicit rate of metastability. We also extract a rate of asymptotic regularity for general normed spaces. Such rates have already been extracted by Kohlenbach and Leu\c{s}tean for different and incomparable conditions on $(\lambda_n)$. The proof analyzed in this paper is also more effective than the proofs treated by Kohlenbach and Leu\c{s}tean in that it does not use Banach limits or weak compactness, which makes the extraction particularly efficient. Moreover, we also give an equivalent axiomatization of uniformly smooth Banach spaces. This paper is part of an ongoing case study of proof mining in nonlinear fixed point theory.
\end{abstract}
\let\thefootnote\relax

\section{Introduction}
Let $H$ be a Hilbert space, $C\subseteq H$ be a convex closed subset and $S:C\to C$ be nonexpansive. For some starting point $x_0\in C$, some anchor $u\in C$ and some sequence $(\alpha_n)\subset\left[0,1\right]$, the Halpern iteration is defined by
\begin{equation}\label{HalpernIteration}
 x_{n+1}:=\alpha_nu+(1-\alpha_n)Sx_n.
\end{equation}
The scheme was first introduced in \cite{Halpern(67)}, albeit only when $C$ is the closed unit ball and $u=0$. For this case, Halpern \cite{Halpern(67)} also gave a set of necessary and a set of sufficient conditions for $(\alpha_n)$ under which the scheme \eqref{HalpernIteration} converges strongly to a fixed point of $S$. However, Halpern's conditions allowed no conclusion whether the natural choice $\alpha_n=1/(n+1)$ is admitted.
Wittmann \cite{Wittmann(92)} answered this question in the affirmative in 1992: If $S$ has a fixed point and the sequence $(\alpha_n)$ satisfies
\begin{multicols}{3}
\begin{enumerate}[label=(\roman{*}), ref=(\roman{*})]
\item\label{i} $\lim\limits_{n\rightarrow\infty}\alpha_n=0$,
\item\label{ii} $\sum\limits_{n=0}^{\infty}\alpha_n=\infty$,
\item\label{iii} $\sum\limits_{n=0}^{\infty}\left\vert\alpha_{n+1}-\alpha_n\right\vert<\infty$,
\end{enumerate}
\end{multicols}
\noindent then the Halpern iteration converges strongly to the fixed point closest to the starting point $x_0$.

Using proof-theoretic methods exhibited in Kohlenbach \cite{Kohlenbach(08),KohlenbachLeustean(12a)}, Leu\c{s}tean \cite{Leustean(07)} extracted from Wittmann's proof a rate of asymptotic regularity for general normed and even hyperbolic spaces, \textit{i.e.}, a rate of convergence for $\left\Vert Tx_n-x_n\right\Vert\to0$ under the assumption that the Halpern iteration remains bounded, which is always the case if $S$ has a fixed point. The rate is highly uniform in the sense that it does not depend on the set $C$, the operator $S$ or the specific choice of the sequence $(\alpha_n)$, but only on witnesses for the existential quantifiers in conditions \ref{i} to \ref{iii} above and a bound on, in essence, the sequence $(x_n)$.

Strong convergence is then established by Wittmann using the metric projection of $x_0$ onto the fixed point set and weak sequential compactness applied to the iteration sequence. As shown by Avigad, Gerhary and Towsner in \cite{AvigadGerhardyTowsner(10)}, there cannot be a computable bound on the rate of convergence even for the special case where $\alpha_n=1/n+1$ and $S$ is linear. In this case, the Halpern iteration coincides with the ergodic average, and so Wittmann's theorem implies von Neumann's mean ergodic theorem.

On the other hand, a uniform rate of metastability (in the sense of Tao \cite{Tao(08),Tao(08a)}) is guaranteed to exist by a general metatheorem of Kohlenbach \cite{Kohlenbach(10)} and was extracted by the same author in \cite{Kohlenbach(11)}. A rate of metastability is a bound on the existential quantifier in the Herbrand normal form of the statement that $\left(x_n\right)$ is Cauchy:
\begin{equation}\label{metastable}
 \forall\varepsilon>0\forall g:\NN\to\NN\exists n\le\Phi(\varepsilon,g)\forall i,j\in\left[n;n+g(n)\right]\Bigl(\left\Vert x_i-x_j\right\Vert<\varepsilon\Bigr),
\end{equation}
where $\left[n;n+g(n)\right]:=\{n,n+1,\ldots,n+g(n)\}$. The bound is highly uniform in the sense that it does not depend on the operator $S$, the starting point $x_0$, the anchor $u$ or the specific Hilbert space. Apart from rates of convergence and divergence for the conditions \ref{i} to \ref{iii}, it only depends on an upper bound on the distance of the starting point from some fixed point of $S$.
In the case $\alpha_n=1/n+1$, Kohlenbach \cite{Kohlenbach(11)} also improved the exponential rate of asymptotic regularity to a quadratic one.
Moreover, these results were also generalised to CAT(0) spaces \cite{KohlenbachLeustean(12)} and CAT($\kappa$) spaces \cite{Leustean(13)}

Closely related to Wittmann's result is the following
\begin{theorem}[Browder \cite{Browder(67a)}]
 Let $H$ be a Hilbert space, $S$ a nonexpansive mapping of $H$ into $H$. Suppose that there exists a bounded closed convex subset $C$ of $H$ mapped by $S$ into itself. Let $u_0$ be an arbitrary point of $C$, and for each $t$ with $0<t<1$, let $S_tx=tSx+(1-t)u_0$.
 
 Then $S_t$ is a strict contraction of $H$ with ratio $t$, $S_t$ has a unique fixed point $z_t$ in $C$, and $z_t$ converges as $t\to1$ strongly in $H$ to a fixed point $v$ of $S$ in $C$. The fixed point $v$ is uniquely specified as the fixed point of $S$ closest to $u_0$.
\end{theorem}

The proof is structured similarly to the proof of Wittmann's theorem in that its ineffective part consists of a projection onto the fixed point set and weak sequential compactness, this time applied to $(z_t)$. In fact, the proof theoretic analysis of Browder's theorem, also carried out in  out in \cite{Kohlenbach(11)}, exhibits interesting parallels to the aforementioned one.

There is also an elementary proof due to Halpern \cite{Halpern(67)} for the special case where $C$ is the closed unit ball of $H$, which can easily be generalised to arbitrary bounded closed convex subsets. The ineffectivity of Halpern's proof stems from the monotone convergence principle, \textit{i.e.}, that every monotone sequence in the real unit interval converges. A metastable version of this can be found on page 30 of \cite{Kohlenbach(08)}. Using this, a simpler rate of metastability was extracted in \cite{Kohlenbach(11)}.

Reich generalised Browder's Theorem to uniformly smooth Banach spaces.
\begin{theorem}[Reich \cite{Reich(80)}]\label{LemmaReich}
 Let $C$ be a closed convex subset of a uniformly smooth Banach space $X$ and let $S:C\to C$ be a nonexpansive mapping with $F\left(S\right)\neq\emptyset$. For $u\in C$, let $\left(z_t\right)$ be defined by the equation $z_t:=tu+\left(1-t\right)Sz_t$. Then
$\left(z_t\right)$ converges strongly to a fixed point of $S$ as $t\to0$.
\end{theorem}
\noindent Recall that a Banach space $X$ is said to be smooth if the limit
\begin{equation*}
 \lim_{t\to0}\frac{\Vert x+ty\Vert-\Vert x\Vert}{t}
\end{equation*}
exists for all $x,y$ in the unit sphere. If the limit is attained uniformly in $x$, then $X$ is said to have a uniformly G\^ateaux differentiable norm. If the limit is attained uniformly in $x,y$, then $X$ is called uniformly smooth. In this case, the normalized duality map $J:X\to X^*$, defined as
\begin{equation*}
 J(x)=\left\{x^*\in X^*:\langle x,x^*\rangle=\Vert x\Vert^2=\Vert x^*\Vert^2\right\}.
\end{equation*}
is single-valued and norm-to-norm uniformly continuous on bounded subsets of $X$. Using Reich's theorem, Shioji and Takahashi \cite{ShiojiTakahashi(97)} generalised Wittmann's result to uniformly smooth Banach spaces. The proof is highly noneffective due to the use of Banach limits, whose existence thus far has only been established making substantial reference to the axiom of choice. This difficulty was overcome by Kohlenbach and Leu\c{s}tean \cite{KohlenbachLeustean(12a)} by the observation that the Banach limits in the proof could be replaced by Cesàro means, which, in turn, are covered by the aforementioned methods.

Xu proved the following variant, which uses neither weak compactness (as in Wittmann's proof) nor Banach limits (as in Shioji and Takahashi's proof).
\begin{theorem}[Xu \cite{Xu(02)}]\label{Xu}
 Let $X$ be a uniformly smooth Banach space, $C$ be a closed convex subset of $X$, and $S:C\to C$ be a nonexpansive mapping with a fixed point. Let $u,x_0\in C$ be given, Assume that $(\alpha_n)\subset[0,1]$ satisfies the control conditions \ref{i}, \ref{ii} and
\begin{enumerate}
 \item[(iii)']$\lim_{n\to\infty}(\alpha_n-\alpha_{n-1})/\alpha=0$
\end{enumerate}
 Then the sequence $\left(x_n\right)$ defined by
 \begin{equation*}
  x_{n+1}=\alpha_nu+(1-\alpha_n)Sx_n
 \end{equation*}
 converges strongly to a fixed point of $S$.
\end{theorem}
Observe that Xu \cite{Xu(02)} showed that Xu's condition (iii)' does not imply \ref{iii}, while Remark 2.3.3 of Schade \cite{Schade(12)} shows that the converse is also not true in general. However, they both cover the most important case $1/n+1$.

We extract two quantitative versions of Theorem \ref{Xu}, namely an explicit and highly uniform rate of convergence for $\left\Vert x_n-Sx_n\right\Vert\to0$ for general Banach spaces $X$, and a rate of metastability $\Phi$ relative to a rate of metastability of the resolvent $\left(z_t\right)$ in the uniformly smooth case (cf. Theorem \ref{LemmaReich}).

As remarked by Kohlenbach and Leu\c{s}tean \cite{KohlenbachLeustean(12a)}, (uniform) smoothness is often only used to ensure that the normalized duality mapping is (uniformly) continuous. This motivates the following
\begin{definition}[Kohlenbach and Leu\c{s}tean \cite{KohlenbachLeustean(12a)}]
 A Banach space $X$ together with a mapping $J:X\to X^*$ satisfying
 \begin{enumerate}[label=(\roman{*}), ref=(\roman{*})]
  \item $\langle x,Jx\rangle=\Vert x\Vert^2=\Vert Jx\Vert^2$ for all $x\in X$ and
  \item $J$ is norm-to-norm uniformly continuous on bounded subsets of $X$
 \end{enumerate}
is called a Banach space with a uniformly continuous duality selection mapping.
\end{definition}

\begin{definition}[Kohlenbach and Leu\c{s}tean \cite{KohlenbachLeustean(12a)}]
Let $X$ be a Banach space be a space with a uniformly continuous duality selection mapping $J$. A map $\omega:\left(0,\infty\right)\times\left(0,\infty\right)\to\left(0,\infty\right)$ is called a modulus of continuity for $J$ if for all $M,\varepsilon>0$,
\begin{equation*}
 \Vert x\Vert,\Vert y\Vert\le M \textmd{ and } \Vert x-y\Vert<\omega(M,\varepsilon)\textmd{ implies } \Vert J(x)-J(y)\Vert<\varepsilon.
\end{equation*}
\end{definition}
\noindent For the $L_p$ spaces ($1<p<\infty$), for example, $\omega$ is calculated explicitly in \cite{KohlenbachLeustean(12a)}. The same paper also shows how to calculate $\omega$ from a modulus of smoothness.

In the appendix, we show that this is actually a characterisation of uniformly smooth Banach spaces: a Banach space is uniformly smooth if and only if it is a Banach space with a uniformly continuous duality selection mapping in which case the duality mapping is single-valued.

Apart from the usual parameters, the rate of metastability will then additionally depend on the modulus of continuity $\omega$. Moreover, in the Hilbert space case, a rate of metastability for the resolvent is known. Our rate of metastability for the Halpern iteration is then of the form
\begin{equation*}
 \left(L(\underline a)\circ g\right)^{\left(B(\underline a)\right)}(0),
\end{equation*}
where $\tilde g(n):=\max\{g(i):i\le n\}+n$, and $g$ is the counterfunction in equation \eqref{metastable}. The functions $B$ and $L$ depend only on the tuple $\underline a$ parameterizing the sequences, the points $u,x_0$ and the sequence $(\alpha_n)$, but \emph{not} on the counterfunction $g$. This was guaranteed a priori by a general metatheorem of Kohlenbach and Safarik \cite{KohlenbachSafarik(13)}, and stems from the fact that the proof only uses a limited amount of the law-of-excluded-middle.

It is still an open problem to extract a rate of metastability for $(z_t)$ in the uniformly smooth case, which would include the $L_p$-spaces (with $1<p<\infty$, $p\ne2$). As already mentioned, this has been done in Hilbert space in \cite{Kohlenbach(11)} and later generalised to the broader class of pseudocontractions in \cite{KoernleinKohlenbach(13)}, however still only for Hilbert space.

\section{Technical Lemmas}
The following classic inequality due to Petryshyn \cite{Petryshyn(70)} is known as the subdifferential inequality.
 \begin{lemma}[Petryshyn{\cite{Petryshyn(70)}}]\label{PolaBana}
  Let $X$ be a real normed space. Then, for all $x,y\in X$ and $j(x+y)\in J(x+y)$, we have $\left\Vert x+y\right\Vert^2\le\Vert x\Vert^2+2\left\langle y,j(x+y)\right\rangle.$
 \end{lemma}
 
\begin{lemma}\label{ineq}
 Let $\left(s_n\right)$ be a sequence of real numbers satisfying
\begin{equation*}
 s_{n+1}\leq\left(1-\alpha_n\right)s_n+\alpha_n\beta_n+\gamma_n,
\end{equation*}
where $\left(\alpha_n\right)\subset\left[0,1\right]$ and $\left(\gamma_n\right)\subset\left[0,\infty\right)$ are real sequences. Moreover, let $\left(\beta_n\right)$ be a real sequence such that for some $m\in\NN$ and all $n\geq m$,
$\beta_n\leq C$. Then
\begin{equation*}
 s_{n+1}\leq\left(\prod_{k=m}^{n}\left(1-\alpha_k\right)\right)s_m+\left(1-\prod_{k=m}^{n}\left(1-\alpha_k\right)\right)C+\sum_{n=m}^{n}\gamma_n,
\end{equation*}
for all $n\geq m$.
\end{lemma}

\begin{proof}
The induction start $n=m$ is clear. Now let $n\geq m$. Then
 \begin{align*}
 s_{n+1}&\leq\left(1-\alpha_n\right)s_n+\alpha_n\beta_n+\gamma_n\\
	&\leq\left(1-\alpha_n\right)\left(\prod_{k=m}^{n-1}\left(1-\alpha_k\right)s_m+\left(1-\prod_{k=m}^{n-1}\left(1-\alpha_k\right)\right)C+\sum_{k=m}^{n-1}\gamma_k\right)+\alpha_n\beta_n+\gamma_n\\
	&\leq\left(\prod_{k=m}^{n}\left(1-\alpha_k\right)\right)s_m+\left(\left(1-\alpha_n\right)-\prod_{k=m}^{n}\left(1-\alpha_k\right)\right)C+\alpha_n\cdot C+\sum_{k=m}^{n}\gamma_k\\
	&=\left(\prod_{k=m}^{n}\left(1-\alpha_k\right)\right)s_m+\left(1-\prod_{k=m}^{n}\left(1-\alpha_k\right)\right)C+\sum_{n=m}^{n}\gamma_n.
\end{align*}
\end{proof}

\begin{lemma}\label{LemmaXuQuant}
 Let $\left(s_n\right)$ be a sequence of non-negative real numbers bounded by some constant $C\in\NN$. Furthermore, let $s_n$ satisfy
\begin{equation*}
 s_{n+1}\leq\left(1-\alpha_n\right)s_n+\alpha_n\beta_n+\gamma_n,\quad\textmd{for all }n\geq0,
\end{equation*}

where $\left(\alpha_n\right)$, $\left(\beta_n\right)$ and $\left(\gamma_n\right)$ are real sequences satisfying
\begin{enumerate}[label=(\roman{*}), ref=(\roman{*})]
 \item $\left(\alpha_n\right)\subset\left[0,1\right]$,
 \item $\forall\varepsilon>0\prod\limits_{k=0}^{S_1(\varepsilon)}\left(1-\alpha_k\right)\leq\varepsilon$
 \item $\forall\varepsilon>0\forall n\geq S_2\left(\varepsilon\right)\beta_n\leq\varepsilon$,
 \item $\left(\gamma_n\right)\subset\left[0,\infty\right)$,
 \item $\forall \varepsilon>0\forall i\geq j\geq S_3\left(\varepsilon\right)\sum\limits_{n=i}^{j}\gamma_n\leq\varepsilon$,
\end{enumerate}
and $S_1$, $S_2$ and $S_3$ are appropriate functions mapping $(0,\infty)\to\NN$. Moreover, suppose that $D:(0,\infty)\to(0,\infty)$ satisfies
\begin{equation*}
 0<D\left(\varepsilon\right)\leq\prod_{n=0}^{\max\left\{S_2\left(\varepsilon/3\right),S_3\left(\varepsilon/3\right)\right\}}\left(1-\alpha_k\right).
\end{equation*}
Then
\begin{equation*}
 \forall\varepsilon>0\forall n\geq\Phi\left(\varepsilon,C,S_1,S_2,S_3,D\right)\left(s_{n+1}\le\varepsilon\right),
\end{equation*}
where
\begin{equation*}
 \Phi\left(\varepsilon,C,S_1,S_2,S_3,D\right)=\max\left\{S_1\left(\frac{\varepsilon\cdot D\left(\varepsilon\right)}{3C}\right),S_2\left(\frac{\varepsilon}{3}\right),S_3\left(\frac{\varepsilon}{3}\right)\right\}, 
\end{equation*}
\end{lemma}

\begin{proof}
 Let $\varepsilon>0$ be given. Set $m:=\max\{R_S(\varepsilon/3),S_3(\varepsilon/3)\}$. Then, for all $n\geq\Phi\ge m$, Lemma \ref{ineq} implies
\begin{align*}
 s_{n+1}&\leq s_m\prod_{k=m}^{n}\left(1-\alpha_k\right)+\frac{\varepsilon}{3}+\frac{\varepsilon}{3}\\
	&= s_m\cdot\frac{\prod\limits_{k=0}^{n}\left(1-\alpha_k\right)}{\prod\limits_{k=0}^{m-1}\left(1-\alpha_k\right)}+\frac{\varepsilon}{3}+\frac{\varepsilon}{3}\leq C\cdot\frac{\varepsilon\cdot D\left(\varepsilon\right)}{3C}\cdot \frac{1}{D\left(\varepsilon\right)}+\frac{2\varepsilon}{3}=\varepsilon.
\end{align*}
\end{proof}

\section{Main Theorems}
The following is essentially Proposition 3.3.8 of \cite{Schade(12)}. We include a proof for completeness.
\begin{proposition}\label{AsymptoticRegularity}
 Let $X$ be a normed space, $C$ be a closed convex subset of X and $S:C\to C$ be a nonexpansive mapping with $F\left(S\right)\neq\emptyset$. Suppose $2\max\{\Vert p-x_0\Vert,\Vert p-u\Vert\}\le M$ for some fixed point $p$ of $S$.
Assume that $\left(\alpha_n\right)\subset\left(0,1\right)$ satisfies
\begin{enumerate}[label=(\roman{*}), ref=(\roman{*})]
 \item\label{I} $\forall\varepsilon>0\forall n\geq R_1\left(\varepsilon\right)\left(\alpha_n\leq\varepsilon\right)$,
 \item\label{II} $\forall\varepsilon>0\prod\limits_{k=1}^{R_2(\varepsilon)}\left(1-\alpha_k\right)\leq\varepsilon$,
 \item\label{III} $\forall\varepsilon>0\forall n\geq R_3\left(\varepsilon\right)\left(\left\vert\alpha_n-\alpha_{n-1}\right\vert\leq\varepsilon\alpha_n\right)$.
\end{enumerate}
Suppose moreover that $D:\RR^+\to\RR^+$ satisfies
\begin{equation*}
 0<D\left(\varepsilon\right)\leq\prod_{k=0}^{R_3\left(\varepsilon/3M\right)}\left(1-\alpha_k\right).
\end{equation*}
Then the sequence $\left(x_n\right)$ generated by
\begin{equation*}
 x_{n+1}:=\alpha_nu+\left(1-\alpha_n\right)Sx_n
\end{equation*}
is asymptotically regular with rate
\begin{align*}
 \psi\left(\varepsilon\right)&:=\max\left\{R_1\left(\frac{\varepsilon}{2M}\right),\Phi\left(\frac{\varepsilon}{2},M,R_2,R_3\left(\frac{\cdot}{M}\right),\mathbf{0},D\right)\right\}
\end{align*}
and $\Phi$ is as in Lemma \ref{LemmaXuQuant}.
\end{proposition}
\begin{proof}
 Observe that
 \begin{align*}
  \left\Vert x_{n+1}-p\right\Vert&=\left\Vert \alpha_n(u-p)+(1-\alpha_n)(Sx_n-p)\right\Vert\\
				  &\le\alpha_n\left\Vert u-p\right\Vert+(1-\alpha_n)\left\Vert x_n-p\right\Vert.
 \end{align*}
 Thus, by induction, $\left\Vert x_n-p\right\Vert\le\max\{\Vert p-x_0\Vert,\Vert p-u\Vert\}\le M/2$. Thus, for all integers $n\ge1$, $\left\Vert x_{n+1}-x_n\right\Vert\le M$ and $\left\Vert u-Sx_{n-1}\right\Vert\le\left\Vert u-p\right\Vert+\left\Vert Sp-Sx_{n-1}\right\Vert\le M$. Hence
\begin{align*}
 \left\Vert x_{n+1}-x_n\right\Vert&=\left\Vert\left(\alpha_n-\alpha_{n-1}\right)\left(u-Sx_{n-1}\right)+\left(1-\alpha_n\right)\left(Sx_n-Sx_{n-1}\right)\right\Vert\\
				  &\leq\left(1-\alpha_n\right)\left\Vert x_n-x_{n-1}\right\Vert+\left\vert\alpha_n-\alpha_{n-1}\right\vert\left\Vert u-Sx_{n-1}\right\Vert\\
				  &\leq\left(1-\alpha_n\right)\left\Vert x_n-x_{n-1}\right\Vert+M\left\vert\alpha_n-\alpha_{n-1}\right\vert\\
				  &=\left(1-\alpha_n\right)\left\Vert x_n-x_{n-1}\right\Vert+\alpha_n\beta_n,
\end{align*}
where
\begin{equation*}
\beta_n:=M\frac{\left\vert\alpha_n-\alpha_{n-1}\right\vert}{\alpha_n}.
\end{equation*}
Lemma \ref{LemmaXuQuant} whith $\gamma_n=0$ then implies that for all $n\geq\Phi\left(\frac{\varepsilon}{2},M,R_2,R_3\left(\frac{\cdot}{M}\right),\mathbf{0},D\right)$, where $\mathbf{0}$ denotes the function that is constant and equal to $0$,
\begin{equation*}
 \left\Vert x_{n+1}-x_n\right\Vert\leq\frac{\varepsilon}{2}.
\end{equation*}
Moreover, $\left\Vert x_{n+1}-Sx_n\right\Vert=\alpha_n\left\Vert u-Sx_n\right\Vert\leq M\alpha_n$. Thus, in total, we get
\begin{equation*}
 \left\Vert x_n-Sx_n\right\Vert\leq\left\Vert x_n-x_{n+1}\right\Vert+\left\Vert x_{n+1}-Sx_n\right\Vert\leq\varepsilon
\end{equation*}
for all $n\geq\psi\left(\varepsilon\right)=\max\left\{R_1\left(\frac{\varepsilon}{2M}\right),\Phi\left(\frac{\varepsilon}{2},M,R_2,R_3\left(\frac{\cdot}{M}\right),\mathbf{0},D\right)\right\}$. In other words, $\psi$ is a rate of asymptotic regularity for $\left(x_n\right)$.
\end{proof}

\begin{corollary}
 If, in the situation of Theorem \ref{AsymptoticRegularity}, $X$ is a normed space, $\alpha_n=1/(n+1)$, $M\ge1$ and $\varepsilon\le3/2$, then
 \begin{equation}\label{psi}
 \psi\left(\varepsilon\right)=\left\lfloor\frac{12M\left\lfloor\frac{3M}{\varepsilon}\right\rfloor}{\varepsilon}\right\rfloor\leq\left\lfloor\frac{36M^2}{\varepsilon^2}\right\rfloor.
\end{equation}
\end{corollary}

\begin{remark}
 Kohlenbach and Leu\c{s}tean \cite{KohlenbachLeustean(12)} also extracted a quadratic rate of asymptotic regularity of Halpern iterates, albeit in CAT(0) spaces under slightly different requirements on $(\alpha_n)$. Since CAT(0) spaces are generalised Hilbert spaces, and the modified conditions on $(\alpha_n)$ also include the natural choice $1/(n+1)$, this corollary
 states an alternative rate of convergence for this special case. In fact, the two rates are almost identical and have the same complexity.
\end{remark}

\begin{theorem}\label{XuQuant}
 Let $X$ be a uniformly smooth Banach space, whose normalized duality mapping $J$ has modulus of uniform continuity $\omega$, $C$ be a closed convex subset of X and $S:C\to C$ be a nonexpansive mapping with $F\left(S\right)\neq\emptyset$ such that $2\max\{\Vert p-x_0\Vert,\Vert p-u\Vert\}\le M$ for some fixed point $p$ of $S$. Suppose that the sequence $\left(z_{\nicefrac{1}{m}}\right)_{m\ge1}$ of Theorem \ref{LemmaReich} is Cauchy with metastability $K$, \textit{i.e.},
\begin{equation*}
 \forall\varepsilon>0\forall g:\NN\to\NN\exists n\leq K\left(\varepsilon,g\right)\forall k,l\in\left[n,n+g\left(n\right)\right]\left\Vert z_{\nicefrac{1}{k}}-z_{\nicefrac{1}{l}}\right\Vert\leq\varepsilon.
\end{equation*}
Assume that $\left(\alpha_n\right)\subset\left(0,1\right)$, $R_1$ to $R_3$ satisfy conditions \ref{I} and \ref{II} of the previous proposition. Moreover, suppose that
\begin{equation*}
 0<E\left(k\right)\leq\prod\limits_{n=0}^{k}\left(1-\alpha_n\right).
\end{equation*}
Then the sequence $\left(x_n\right)$ generated by
\begin{equation*}
 x_{n+1}:=\alpha_nu+\left(1-\alpha_n\right)Sx_n
\end{equation*}
converges strongly to fixed point of $S$. Moreover,
\begin{equation}\label{meta}
 \forall\varepsilon>0\forall g:\NN\to\NN\exists n\leq \Sigma\left(\varepsilon,g,M,K,E,R_1,R_2,R_3,\omega\right)\forall k,l\in\left[n,n+g\left(n\right)\right]\Bigl(\left\Vert x_k-x_l\right\Vert\leq\varepsilon\Bigr),
\end{equation}
where $\Sigma:=\max\left\{R_2\left(\frac{E\left(k\right)\cdot\varepsilon^2}{12M^2}\right),\,\tilde \Gamma\leq k\leq \Gamma\right\}$ and
\begin{align*}
 \varepsilon_0&:=\min\left\{\delta,\omega\left(M,\delta\right)\right\},\quad\delta:=\frac{\varepsilon^2}{144M},\quad D\left(\varepsilon\right):=E\left(R_3\left(\varepsilon/3M\right)\right),\\
 f^*\left(k\right)&:=f\left(k+m_0\right)+m_0,\quad m_0:=\left\lceil\frac{72M^2}{\varepsilon^2}\right\rceil,\quad\tilde E\left(k\right):=E\left(\varphi(k)\right),\\
 f\left(k\right)&:=\left\lceil\max\left\{24M^2\cdot\frac{\max\left\{g^*\left(R_2\left(\frac{\tilde E\left(k\right)\cdot\varepsilon^2}{12M^2}\right)\right),\varphi(k)+1\right\}-\varphi(k)-1}{\varepsilon^2}-k,0\right\}\right\rceil,\\
 g^*\left(k\right)&:=k+g\left(k\right),\quad\varphi(k):=\psi\left(\frac{\varepsilon^2}{72Mk}\right)\\
 \Gamma&:=\max\left\{\varphi(k),\,m_0\leq k\leq K\left(\varepsilon_0,f^*\right)+m_0\right\},\\
 \tilde\Gamma&:=\min\left\{\varphi(k),\,m_0\leq k\leq K\left(\varepsilon_0,f^*\right)+m_0\right\}
\end{align*}
and $\psi$ is as defined in Proposition \ref{AsymptoticRegularity}, \textit{i.e.}, a rate of asymptotic regularity for $(x_n)$.
\end{theorem}

\begin{proof}
Set $z_m:=z_{1/m}$. First observe that, as before, $\Vert x_n-p\Vert\le M/2$. Moreover,
\begin{equation*}
 \Vert z_m-p\Vert=\Vert \frac{1}{m}(u-p)+(1-\frac{1}{m})(Sz_m-Sp)\Vert\le \frac{1}{m}\Vert u-p\Vert+(1-\frac{1}{m})\Vert z_m-p\Vert.
\end{equation*}
Therefore, $\Vert z_m-p\Vert\le M/2$ for all $m$.

Now let $\varepsilon>0$ and $g:\NN\to\NN$ be given. Then there is some $K_1\leq K\left(\varepsilon_0,f^*\right)$ such that $\left\Vert z_k-z_l\right\Vert\leq\varepsilon_0$ for all $k,l\in\left[K_1,K_1+f^*\left(K_1\right)\right]$.
Set $K_0:=m_0+K_1\leq m_0+K\left(\varepsilon_0,f^*\right)$. Then, the interval $I:=\left[K_0,K_0+f\left(K_0\right)\right]=\left[K_1+m_0,K_1+m_0+f\left(K_1+m_0\right)\right]\subseteq\left[K_1,K_1+f^*\left(K_1\right)\right]$ and so we have both $\left\Vert z_m-z_{K_0}\right\Vert\leq\varepsilon_0$ for all $m\in I$, and $K_0\geq m_0\geq72M^2/\varepsilon^2$. Consequently, if we let
\begin{equation*}
 \beta_n^m:=2\left\langle u-z_m,J(x_n-z_m)\right\rangle-\frac{M^2}{m},
\end{equation*}
 then, since $\Vert z_{K_0}-z_m\Vert\le\varepsilon_0\le\omega(M,\delta)$, we see that $\left\langle u-z_m,J\left(x_n-z_m\right)-J\left(x_n-z_{K_0}\right)\right\rangle\le\Vert u-z_m\Vert\cdot\Vert J(z_{K_0}-z_m)\Vert\le 2M\delta$ for all $m\in I$,
\begin{align*}
\beta_n^m-\beta_n^{K_0}&=2\left\langle u-z_m,J\left(x_n-z_m\right)\right\rangle-2\left\langle u-z_{K_0},J\left(x_n-z_{K_0}\right)\right\rangle+\left(\frac{1}{K_0}-\frac{1}{m}\right)M^2\\
		       &\leq2\left\langle u-z_m,J\left(x_n-z_m\right)-J\left(x_n-z_{K_0}\right)\right\rangle+2\left\langle z_{K_0}-z_m,J\left(x_n-z_{K_0}\right)\right\rangle+\frac{M^2}{K_0}\\		       
		       &\leq\frac{\varepsilon^2}{72M}\cdot M+\frac{\varepsilon^2}{72M}\cdot M+\frac{\varepsilon^2}{72M^2}\cdot M^2=\frac{\varepsilon^2}{24}
\end{align*}
for all $m\in I$. Moreover, by the subdifferential inequality, we obtain analogously to \cite{Xu(02)},
 \begin{align*}
  \left\Vert z_m-x_n\right\Vert^2&\le(1-\frac{1}{m})^2\Vert Sz_m-x_n\Vert^2+\frac{2}{m}\left\langle u-x_n,J(z_m-x_n)\right\rangle\\
				  &\le(1-\frac{1}{m})^2\left(\Vert Sz_m-Sx_n\Vert+\Vert Sx_n-x_n\Vert\right)^2\\
				  &\quad+\frac{2}{m}\left(\Vert z_m-x_n\Vert^2+\langle u-z_m,J(z_m-x_n)\rangle\right)\\
				  &\le\left(1+\frac{1}{m^2}\right)\Vert z_m-x_n\Vert^2+\Vert Sx_n-x_n\Vert\left(2\Vert z_m-x_n\Vert+\Vert Sx_n-x_n\Vert\right)\\
				  &\quad+\frac{2}{m}\left\langle u-z_m,J(z_m-x_n)\right\rangle.
 \end{align*}
 Therefore, 
\begin{align*}
 2\left\langle u-z_m,J\left(x_n-z_m\right)\right\rangle&\leq \frac{1}{m}\left\Vert z_m-x_n\right\Vert^2+m\left\Vert Sx_n-x_n\right\Vert\left(2\left\Vert z_m-x_n\right\Vert+\left\Vert Sx_n-x_n\right\Vert\right)\\
	&\leq \frac{M^2}{m}+3Mm\left\Vert x_n-Sx_n\right\Vert,
\end{align*}
so $\beta_n^m\le3mM\Vert Sx_n-x_n\Vert$. Since $\psi$ is a rate of asymptotic regularity for $x_n$, we know that $\beta_n^{K_0}\leq\varepsilon^2/24$ for all $n\geq n_0:=\psi\left(\frac{\varepsilon^2}{72MK_0}\right)=\varphi(K_0)$ and so
$\beta_n^m\leq\frac{\varepsilon^2}{24}+\beta_n^{K_0}\leq\varepsilon^2/12$ for all $n\geq n_0$ and $m\in I$. Consequently, applying the subdifferential inequality twice yields
\begin{align*}
 \left\Vert x_{n+1}-z_m\right\Vert^2&=\left\Vert\left(1-\alpha_n\right)\left(Sx_n-z_m\right)+\alpha_n\left(u-z_m\right)\right\Vert^2\\
		&\leq\left(1-\alpha_n\right)^2\left\Vert Sx_n-z_m\right\Vert^2+2\alpha_n\left\langle u-z_m,J\left(x_{n+1}-z_m\right)\right\rangle\\
		&\leq\left(1-\alpha_n\right)^2\left\Vert Sx_n-Sz_m+Sz_m-z_m\right\Vert^2+\alpha_n\beta_{n+1}^m+\frac{\alpha_nM^2}{m}\\
		&\leq\left(1-\alpha_n\right)^2\left(\left\Vert Sx_n-Sz_m\right\Vert^2+2\left\langle Sz_m-z_m,J\left(Sx_n-z_m\right)\right\rangle\right)+\alpha_n\beta_{n+1}^m+\frac{\alpha_nM^2}{m}\\
		&\leq\left(1-\alpha_n\right)^2\left(\left\Vert x_n-z_m\right\Vert^2+\frac{2}{m}\left\langle Sz_m-u,J\left(Sx_n-z_m\right)\right\rangle\right)+\alpha_n\beta_{n+1}^m+\frac{\alpha_nM^2}{m}\\
		&\leq\left(1-\alpha_n\right)\left\Vert x_n-z_m\right\Vert^2+\alpha_n\beta_{n+1}^m+\frac{2M^2}{m}.
\end{align*} 
Thus, if we apply Lemma \ref{ineq} with $\gamma_n=\frac{2M^2}{m}$, we obtain for $n>n_0$
\begin{align}\label{ineq2}
 \left\Vert x_{n}-z_m\right\Vert^2&\leq\left(\prod_{k=n_0}^{n-1}\left(1-\alpha_k\right)\right)\left\Vert x_{n_0}-z_m\right\Vert^2+\left(1-\prod_{k=n_0}^{n-1}\left(1-\alpha_k\right)\right)\frac{\varepsilon^2}{12}+\frac{2M^2}{m}\left(n-1-n_0\right)\notag\\
				    &\leq\left(\prod_{k=n_0}^{n-1}\left(1-\alpha_k\right)\right)M^2+\frac{\varepsilon^2}{12}+\frac{2M^2}{m}\left(n-1-n_0\right).
\end{align}
Therefore, for all $n\geq n_1:=\max\left\{R_2\left(\frac{E\left(n_0\right)\cdot\varepsilon^2}{12M^2}\right),n_0+1\right\}$,
\begin{equation}\label{ineq3}
 \left\Vert x_{n}-z_m\right\Vert^2\leq\frac{\varepsilon^2}{6}+\frac{2M^2}{m}\left(n-1-n_0\right).
\end{equation}
Now observe that
\begin{align*}
 K_0+f\left(K_0\right)&=\left\lceil\max\left\{24M^2\cdot\frac{g^*\biggl(\max\left\{R_2\left(\frac{\tilde E\left(K_0\right)\cdot\varepsilon^2}{12M^2}\right),n_0\right\}\biggr)-\varphi(K_0)-1}{\varepsilon^2},K_0\right\}\right\rceil\\
		      &\geq24M^2\cdot\frac{g^*\biggl(\max\left\{R_2\left(\frac{E\left(n_0\right)\cdot\varepsilon^2}{12M^2}\right),n_0\right\}\biggr)-\varphi(K_0)-1}{\varepsilon^2}\\
		      &=24M^2\cdot\frac{g^*\left(n_1\right)-n_0-1}{\varepsilon^2}.
\end{align*}
Thus, if we set $P:=K_0+f\left(K_0\right)\in I$, we have
\begin{equation*}
 P\geq \frac{24M^2\left(g\left(n_1\right)+n_1-1-n_0\right)}{\varepsilon^2}.
\end{equation*}
Then, for all $n\in\left[n_1,n_1+g\left(n_1\right)\right]$
\begin{align*}
 \frac{2M^2}{P}\left(n-1-n_0\right)&\leq\frac{2M^2\left(n_1+g\left(n_1\right)-1-n_0\right)}{2\cdot12M^2\left(g\left(n_1\right)+n_1-1-n_0\right)}\varepsilon^2=\frac{\varepsilon^2}{12}.
\end{align*}
Then \eqref{ineq3} yields
\begin{equation*}
 \left\Vert x_n-z_P\right\Vert\leq\frac{\varepsilon}{2}.
\end{equation*}
Thus, for all $k,l\in\left[n_1,n_1+g\left(n_1\right)\right]$,
\begin{equation*}
 \left\Vert x_k-x_l\right\Vert\leq\left\Vert x_k-z_P\right\Vert+\left\Vert x_l-z_P\right\Vert\leq\varepsilon.
\end{equation*}
The claim follows from the observation that, since $m_0\le K_0\le m_0+K(\varepsilon_0,f^*)$, $n_0\in[\tilde\Gamma,\Gamma]$ and so $n_1\leq\Sigma\left(\varepsilon,g,M,R_1,R_2,R_3,\omega\right)$.
\end{proof}

As mentioned in the introduction, Kohlenbach and Leu\c{s}tean \cite{KohlenbachLeustean(12a)} extracted from a proof due to Shioji and Takahashi \cite{ShiojiTakahashi(97)} a bound for slightly different conditions on $(\alpha_n)$, which also include $\alpha_n=1/n$. If, furthermore, we restrict ourselves to a Hilbert space setting, then the metastability of the resolvent $(z_n):=(z_{\nicefrac{1}{n}})$ is known from the following
\begin{theorem}[Kohlenbach \cite{Kohlenbach(11)}]
 Let $X$ be a real Hilbert space and $C\subset X$ be a bounded closed convex subset with diameter $d_C\leq M$ and $S:C\to C$ be a nonexpansive mapping.
Then for all $\varepsilon>0$ and $g:\NN\to\NN$,
\begin{equation*}
 \exists n\leq K(\varepsilon,g):=\tilde g^{(\lceil M^2/\varepsilon^2\rceil)}\left(0\right)\forall i,j\in\left[n,n+g\left(n\right)\right]\bigl(\left\Vert z_i-z_j\right\Vert\bigr)
\end{equation*}
and $\tilde g\left(n\right)=\max\left\{n,g\left(n\right)\right\}$.
\end{theorem}
\noindent Moreover, it is obvious that we may take $\omega=\operatorname{id}$ in this case. After lengthy, but trivial calculations, we see that the rate of metastability in our case is obtained from the counterfunction $g$ is modified to, essentially
\begin{equation*}
 f^*(k)=\frac{M^2}{\varepsilon^2}g\left(\frac{M^6k^2}{\varepsilon^6}\right)
\end{equation*}
and then iterated $M^3/\varepsilon^4$ many times before being multiplied by $M^6/\varepsilon^6$. In the Addendum \cite{KohlenbachLeusteanAddendum(12)} to \cite{KohlenbachLeustean(12a)}, the counterfunction is modified to essentially
\begin{equation*}
 f^*(k)=\frac{M^2}{\varepsilon^2}g\left(\frac{M^6k^2}{\varepsilon^4}\right),
\end{equation*}
which is slightly better, but here $f^*$ is iterated $M^4/\varepsilon^4$ times.

\section{Appendix}

  In this section, we show that a Banach space is smooth if and only if it has a norm-to-norm uniformly continuous duality selection mapping. It is well-known that if $X$ is a uniformly smooth Banach space, then the normalized duality map is single valued and norm-to-norm uniformly continuous. Thus we only need to show the converse direction. As a corollary, one then immediately concludes that a Banach space has at most one uniformly continuous duality selection mapping.
 
 The proof of the following theorem follows closely the work of Giles \cite{Giles(67)}.
 \begin{theorem}
 Let $X$ be a normed space with norm $\Vert \cdot\Vert$. If $X$ is a space with a norm-to-norm (uniformly) continuous duality selection map, it is (uniformly) smooth.
 \end{theorem}

 \begin{proof}
   For $x,y\in X$, $\Vert x\Vert,\Vert y\Vert=1$, and real $\lambda>0$,
   \begin{align*}
    \frac{\left\Vert x+\lambda y\right\Vert-\left\Vert x\right\Vert}{\lambda}&=\frac{\left\Vert x+\lambda y\right\Vert\left\Vert x\right\Vert-\left\Vert x\right\Vert^2}{\lambda\Vert x\Vert}\\
									       &\ge\frac{\left\langle x+\lambda y,J(x)\right\rangle-\Vert x\Vert^2}{\lambda\Vert x\Vert}\\
									       &=\frac{\Vert x\Vert^2+\lambda\langle y,J(x)\rangle-\Vert x\Vert^2}{\lambda\Vert x\Vert}\\
									       &=\frac{\langle y,J(x)\rangle}{\Vert x\Vert}.
   \end{align*}
   On the other hand,
   \begin{align*}
    \frac{\left\Vert x+\lambda y\right\Vert-\left\Vert x\right\Vert}{\lambda}&=\frac{\Vert x+\lambda y\Vert^2-\Vert x\Vert\Vert x+\lambda y\Vert}{\lambda\Vert x+\lambda y\Vert}\\
									       &\le\frac{\langle x+\lambda y,J(x+\lambda y)\rangle-\langle x,J(x+\lambda y)\rangle}{\lambda\Vert x+\lambda y\Vert}\\
									       &=\frac{\langle y,J(x+\lambda y)\rangle}{\Vert x+\lambda y\Vert}.
   \end{align*}
  In total,
  \begin{equation*}
   \frac{\langle y,J(x)\rangle}{\Vert x\Vert}\le\frac{\Vert x+\lambda y\Vert-\Vert x\Vert}{\lambda}\le\frac{\langle y,J(x+\lambda y)\rangle}{\Vert x+\lambda y\Vert}.
  \end{equation*}
  Analogously, for real $\lambda<0$,
  \begin{equation*}
   \frac{\langle y,J(x)\rangle}{\Vert x\Vert}\ge\frac{\Vert x+\lambda y\Vert-\Vert x\Vert}{\lambda}\ge\frac{\langle y,J(x+\lambda y)\rangle}{\Vert x+\lambda y\Vert},
  \end{equation*}
  whence the (uniform) continuity of $J$ implies the (uniform) smoothness.
  \end{proof}

\textbf{Acknowledgements:} I am very grateful to Ulrich Kohlenbach for providing valuable input at various places of this paper.

\bibliographystyle{abbrv}
\bibliography{../FixedPointTheory.bib}{}
\end{document}